\documentclass[a4paper,11pt]{article}

\usepackage[utf8]{inputenc}
\usepackage{tikz}
\usetikzlibrary{matrix}
\usepackage[all]{xy}
\usetikzlibrary{patterns}
\usepackage[english]{babel}
\usepackage{amsmath,mathtools,bm}
\usepackage{amssymb}
\usepackage{amsthm}
\usepackage{verbatim}
\usepackage{graphicx}
\usepackage[small]{caption}
\usepackage[all]{xy}

\usepackage[numbers]{natbib}
\usepackage{xurl}
\usepackage{enumitem}

\usepackage{tikz-cd }             
\usetikzlibrary{patterns}
\usepackage{graphicx}

\theoremstyle{definition}
\newtheorem{definizione}{Definition}[section]
\newtheorem{remark}[definizione]{Remark}
\theoremstyle{plain}
\newtheorem{teorema}[definizione]{Theorem}
\newtheorem*{teorema*}{Theorem}
\theoremstyle{plain}
\newtheorem{lemma}[definizione]{Lemma}
\newtheorem*{property*}{Property 1}
\theoremstyle{plain}
\newtheorem{corollario}[definizione]{Corollary}
\theoremstyle{plain}

\theoremstyle{definition}
\newtheorem{example}[definizione]{Example}
\theoremstyle{plain}

\theoremstyle{plain}

\newcommand{\numberset}{\mathbb}
\newcommand{\N}{\numberset{N}}
\newcommand{\f}{\longrightarrow}
\newcommand{\cbs}[1]{{\mathcal{C}_{BS}(#1)}}

\newcommand{\edges}[1]{{\mathcal E}(#1)}
\newcommand{\partitions}[1]{{\mathcal P}(#1)}
\newcommand{\field}{{\mathbb F}}
\newcommand{\ZZ}{{\mathbb Z}}
\newcommand{\QQ}{{\mathbb Q}}
\newcommand{\kk}{{\mathbf k}}
\newcommand{\innerproduct}[2]{{\left< #1 , #2 \right>}}

\usepackage{subcaption}

\newcommand\ext{\mathrm{ext}}
\DeclareMathOperator\id{id}
\newcommand\Conf{\mathrm{Conf}}

\newcommand\conn{\mathrm{con}}
\newenvironment{AllowDisplayBreaks}{\allowdisplaybreaks}{\ignorespacesafterend}
\newenvironment{salign*}{\begin{equation*}\begin{aligned}}{
    \end{aligned}\end{equation*}\ignorespacesafterend}
\usepackage{csgbedit}
\usepackage{blindtext}
\title{Graph cohomologies and rational homotopy type of configuration spaces}	\author{Marcel B\"{o}kstedt, Erica Minuz\footnote{The authors are based in Aarhus University, Science and Thechnology, Department of Mathematics. The work is supported by \textit{Det frie forskningsr\aa{}d, natur og univers} $6108-00539A$.}}
\date\today

\begin{document}
	\maketitle
	We compare the cohomology complex defined by Baranovsky and Sazdanovi\'{c}, that is the $E_{1}$ page of a spectral sequence converging to the homology of the configuration space depending on a graph, with the rational model for the configuration space given by Kriz and Totaro. In particular we generalize the rational model to any graph and to an algebra over any field. We show that, in the case of configuration spaces of point on a even dimensional manifold, the dual of the Baranovsky and Sazdanovi\'{c}'s complex is quasi equivalent to this generalized version of the Kriz's model.
	\tableofcontents	
	\section{Introduction}
In $2006$, in the study of the chromatic polynomial, \textit{Eastwood} and \textit{Huggett} define a generalized configuration space of points in a manifold, depending of a graph \cite{E-H}. Let $M$ is a manifold and $\Gamma$ a graph with set of vertices $V=\{v_{1},..,v_{n}\}$ and set of edges $E(\Gamma)$. The configuration space of $n$ points in $M$, $\Conf(M,\Gamma)$ is defined as 
\begin{equation*}
\Conf(M,\Gamma)=\{(x_{1},\dots,x_{n})\in M^{n}; x_{i}\neq x_{j} \text{ if }e_{i,j}\in E(\Gamma)\}.
\end{equation*} When $\Gamma$ is a complete graph, $\Conf(M,\Gamma)$ coincide with the classical configuration space of $n$ points in a manifold:	\begin{equation*}
\Conf(M,n)=\{(x_{1},\dots,x_{n})\in M^{n}; x_{i}\neq x_{j} \text{ for } i\neq j\}.
\end{equation*}  In $2012$ \textit{Baranovsky} and \textit{Sazdanovi\'{c}} \cite{B-S} define a graph cohomology complexinspired by the one defined by \textit{Helme-Guizon} and \textit{Rong}  in \cite{H-G}. We denote this complex by $\mathcal{C}_{BS}$. They prove the existence a spectral sequence with $E_{1}$ page given $\mathcal{C}_{BS}$ converging to the homology of the \textit{Eastwood} and \textit{Huggett's} configuration space \cite{B-S}. The spectral sequence for the case where the graph is the complete graph was given in $1991$ by \textit{Bendersky} and \textit{Gitler} \cite{BG}.	

The cohomology and the rational homotopy type of $\Conf(M,n)$ has been studied before. If the manifold is $\mathbb{R}^d$, the results about the cohomology of $\Conf(\mathbb{R}^d,n)$ is due to Arnold
\cite{Arnol'd1969} in the case $r=2$ and Cohen
\cite{COHEN199519} for $r\geq 3$. In $1994$ \emph{Fulton} and \emph{MacPherson}
\cite{FM} constructed a model for the rational homotopy type of
$\Conf(X,n)$, where $X$ is a non singular, compact, complex variety.  This
model depends on the cohomology ring $H^{\ast}(X,\mathbb{Q})$, the
orientation and the Chern classes. The same year, \emph{Kriz} in \cite{Kriz} described a differential
graded algebra $E[n]$ that is a rational model for $\Conf(X,n)$ and that
is independent from the Chern classes. 
$E[n]$ was described in the same time in the work by \emph{Totaro}
\cite{TOTARO} and it appeared to be isomorphic to the $E_{2}$ page of
the Larey spectral sequence of the inclusion
$\Conf(X,n)\hookrightarrow X^{n}$.  The algebra $E[n]$ will be here discussed in details in \emph{Theorem}~\ref{kriz}. Later, \emph{Lambecht} and \emph{Stanley} studied the rational models
for configurations spaces where $X$ is a simply connected closed
manifold. They showed that a simply connected closed manifold always
admits a Poincaré duality model $A$ \cite{ASENS_2008_4_41_4_497_0}. In
$2004$ they described the case $k=2$, a configuration space of $2$
points in a manifold \cite{Lambrechts2004} and they defined a model for
its rational homotopy type. In $2008$ they
presented a potential model for the general case
\cite{lambrechts2008}. This commutative graded algebra is denoted by  $G_A(n)$. They conjectured that if $X$ is a simply
connected $m$-manifold, $G_A(n)$ is a rational model for $\Conf(X,n)$.
%
%
%
In $2019$ \emph{Idrissi} \cite{idrissi:hal-01438861} proved the
conjecture true for the real homotopy type and manifolds of dimension
at least $4$. \emph{Campos} and \emph{Willwacher} few years before
\cite{CRT} constructed a real model for the configuration space of
points in a manifold.

In this article, we compare the graph cohomology complex
	$\mathcal{C}_{BS}(\Gamma)$ defined by \emph{Baranovsky} and
	\emph{Sazdanovi\'{c}} in \cite{B-S} and described in \emph{Section
	}\ref{BS}, with the model for the rational homotopy type given by
	\emph{Kriz} and \emph{Totaro} denoted by $E[n]$. In the second and third sections we give the definition of these two cohomology complexes. In the following section, we describe Frobenius algebras and
	give some technical results that will be later used. In \emph{Section
	}\ref{dual} we define the dual of the complex
	$(\mathcal{C}_{BS}(\Gamma),\partial)$ that we will call
	$(\mathcal{C}_{BS}(\Gamma)^{\ast},\delta)$.  The complex
	$(\mathcal{C}_{BS}(\Gamma),\partial)$ is the $E_{1}$ page of a spectral
	sequence converging to the relative cohomology
	$H^{\ast}(M^{n}, Z_{\Gamma}, R)$, and if the space $M$ is a compact oriented manifold of
	dimension $m$ the cohomology is isomorphic to the homology
	$H_{mn-\ast}(\Conf(M,\Gamma), R)$. In this case the dual complex
	$(\mathcal{C}_{BS}^{\ast}(\Gamma),\delta)$ converges to the cohomology
	of the configuration space $H^{mn-\ast}(\Conf(M,\Gamma), R)$. On the other hand, the
	cohomology of the complex $E[n]$ is the cohomology of the
	configuration space.  By \emph{Remark} \ref{remark3}, if $M$ is a
	compact K\"{a}hler manifold and the coefficient ring is $R=\mathbb{Q}$,
	the spectral sequence degenerates at page $E_{2}$. In this case the
	two complexes are quasi equivalent. In the following sections we prove
	that there is in general a quasi equivalence between
	$C_{BS}(\Gamma)^{\ast}$ and a generalized version of $E[n]$, called $R(\Gamma,A)$. In the definition of this generalised complex, a ring $\Delta[G_{a,b}]/\sim$
	is involved. This is the exterior algebra over the generator
	corresponding to the edges in the graph $\Gamma$ quotient by a
	relation, that we call the \emph{generalised Arnold relation}. We
	denote this ring by $R(\Gamma)$. \emph{Section }\ref{ring} describes
	it for a complete graph $K_n$, and in this case the relation is
	the usual Arnold relation. We will call it $R(K_n)$. In the following
	section we define $R(\Gamma,A)$ for an even dimensional manifold. It depends on a graph $\Gamma$ not
	necessarily complete, and a Frobenius algebra $A$ over
	any field. In the case of an even dimensional formal manifold and a complete graph $\Gamma$, $R(\Gamma,A)$ coincide with the CDGA that \textit{Idrissi} \cite{idrissi:hal-01438861} proves to be a real model for $\Conf(M,n)$. \emph{Section
	}\ref{section quasi eq} contains the main theorem of the chapter.
	\begin{teorema*}
		Let $S\subseteq\Gamma$. The map
		\begin{align*}
		F: \cbs \Gamma^* &\to R(\Gamma, A)\\
		F(G_{S} \otimes x) &= [G_{S}\otimes x ]
		\end{align*}
		is a quasi equivalence.
	\end{teorema*}
	In \cite{cs_and_masseyTF} \emph{Thomas} and \emph{Felix} prove that
	the $mn$ suspension of the $E_{2}$ term of the \emph{Bendersky-Gitler}
	spectral sequence is isomorphic to the $E_{2}$ term of of the
	\emph{Cohen} and \emph{Taylor} spectral sequence of which the
	\emph{Kriz}'s model is a special case. Our theorem presents an
	alternative proof and generalization of this result.
	In the last section we discuss the chain complex
	$\mathcal{C}_{BS}(\Gamma)^{\ast}/I(\Gamma)$, where $I(\Gamma)$ is the
	ideal generated by the generalised Arnold relation and we show that it
	is isomorphic to $R(\Gamma, A)$.\\
	
	\textbf{Acknowledgments.} We would like to thank Najib Idrissi and Søren Galatius for reading the second author's Ph.D. thesis, of which this article was part. The second author is especially grateful for Galatius very detailed corrections. Finally, we thank Thomas Willwacher for fruitful conversations.
	
	\section{Baranovsky-Sazdanovi\'{c}'s graph cohomology}\label{BS}

	This section presents the graph complex defined by \emph{Baranovsky}
	and \emph{Sazdanovi\'{c}} in \cite{B-S}. Their definition is inspired
	by the work by \emph{Helme-Guizon} and \emph{Rong} \cite{H-G}, whose
	construction develops from the cohomology theory defined by
	\emph{M. Khovanov} in \cite{khovanov}. There he associates to each
	link a family of cohomology groups whose Euler characteristic is the
	Jones polynomial of the link.  \emph{Helme-Guizon's} and \emph{Rong's}
	graph cohomology expands the \emph{Khovanov's} definition associating
	to each graph, graded cohomology groups whose Euler characteristic is
	the chromatic polynomial of the graph. \emph{Baranovsky} and
	\emph{Sazdanovi\'{c}} in \cite{B-S} prove that there is a spectral
	sequence that relates the graph cohomology defined by
	\emph{Helme-Guizon} and \emph{Rong} with the cohomology of
	configuration spaces, verifying a conjecture posed by \emph{Khovanov}.
	
	We give here the definition of the graph cohomology complex. We refer
	to \cite{B-S} for the definitions and the notation with the exception
	of the notation of the complex that we will call
	$\mathcal{C}_{BS}(\Gamma)$. Then, we state some results relating the
	complex to the homology of configuration spaces.
	
	Let $A$ be a graded commutative algebra over a commutative ring $R$, and 
	assume that $A$ is a projective $R$-module. Let $\Gamma$ be a finite
	graph, $V=V(\Gamma)$ be the set of vertices and $E(\Gamma)$ the set of
	edges. We choose an order on the vertices. This gives an orientation
	on every edge $\alpha$ in $E(\Gamma)$, if $\alpha$ connects the
	vertices $i$ and $j$ and $i\leq j$, $\alpha:i\rightarrow j$. For any
	subset $S$ of $E(\Gamma)$, we denote by $[\Gamma:S]$ the subgraph that
	has as vertices the same vertices of $\Gamma$ and as edges the edges
	in $S$, we denote by $l(S)$ the number of connected components of
	$[\Gamma:S]$.
	\begin{definizione}
		[\cite{B-S}]\label{def Cbs} Let $\Lambda$ be an exterior algebra
		over $R$ with generators $e_{\alpha}$, $\alpha\in E(\Gamma)$, and
		$e_{S}$ be the exterior product of $e_{\alpha}$, $\alpha\in S$,
		ordered with the lexicographic order of the pair $(i,j)$ where
		$\alpha:i\rightarrow j$.
		
		The bigraded complex, that we will here denote by
		$\mathcal{C}_{BS}(\Gamma)$, is defined as
		\begin{equation*}
		\mathcal{C}_{BS}(\Gamma)=\Lambda\otimes A^{\otimes n}/e_{\alpha}\otimes (a[i]-a[j]),
		\end{equation*}
		the algebra $\Lambda\otimes A^{\otimes n}$ quotient by the relation
		$e_{\alpha}\otimes (a[i]-a[j])$, where $a\in A$,
		$\alpha:i\rightarrow j\in E(\Gamma)$ and $a[i]$ denotes the element
		$1^{\otimes i-1}\otimes a\otimes 1^{\otimes n-i}\in\nobreak
		A^{\otimes n}$.  The complex has a bigrading given by the sum of the
		grading of the elements $e_{\alpha}$ of bidegree $(0,1)$ and the
		elements $1\otimes a_{1}\otimes\dots\otimes a_{n}$ with bidegree
		$(\smash{\sum_{i=1}^{n}}\deg_{A}a_{i},0)$, so the degree of 
		$e_{S}\otimes a_{1}\otimes\dots\otimes a_{n}$ in $\mathcal{C}_{BS}(\Gamma)$ is
		$(\smash{\sum_{i=1}^{n}}\deg_{A}a_{i},
		|S|)$.
		
		The differential of degree $(0,1)$ is given by the exterior product
		\begin{equation*}
		\partial=\sum_{\alpha\in E(\Gamma)}e_{\alpha}
		\end{equation*} 
	\end{definizione}
	\begin{remark}
		The assumption of $A$ be a projective $R$-module is used in the
		proof of the convergence of the spectral sequence. We refer to
		\cite{B-S} for the definition of the spectral sequence and the
		proof.
	\end{remark}
	\begin{remark}
		The complex $C_{BS}(\Gamma)$ is isomorphic
		to \begin{equation*}\bigoplus_{S\subseteq E(\Gamma)}e_{S}\otimes
		A^{l(S)}.\end{equation*} That is for every $n\in \mathbb{N}$
		\begin{equation*}C_{BS}^{n}(\Gamma)=\bigoplus_{S\subseteq E(\Gamma),
			|S|=n}e_{S}\otimes A^{l(S)}.\end{equation*} For
		$S\subset E(\Gamma)$, each term $a_{i}$ of the element
		$e_{S}\otimes a_{1}\otimes\dots\otimes a_{l(S)}\in A^{l(S)}$,
		corresponds to a component in $[\Gamma:S]$. In the case $S=\emptyset$,
		the components of $[\Gamma:\emptyset]$ are the vertices in $\Gamma$.
		We can construct a map
		\begin{equation*}
		\phi : (\Lambda\otimes A^{\otimes n}/{\sim})\rightarrow\bigoplus_{S\subseteq E(\Gamma)}e_{S}\otimes A^{l(S)}
		\end{equation*}
		such that if $\alpha$ is an edge in $E(\Gamma)$,
		$\alpha:i\rightarrow j$, then
		\begin{equation*}
		\phi(e_{\alpha}\otimes a_{1}\otimes\dots\otimes a_{n})=e_{\alpha}\otimes a_{1}\otimes\dots\otimes (-1)^{\tau} a_{i}a_{j}\otimes\dots\otimes a_{l(\alpha)}.
		\end{equation*}
		The terms $a_{i}$ and $a_{j}$ are multiplied with a sign that is the
		Kozul sign given by the permutation in the tensor product that
		brings $a_{j}$ close to $a_{i}$, here $l(\alpha)=n-1$. The inverse
		is given by
		\begin{align*}
		\MoveEqLeft \phi^{-1}(e_{\alpha}\otimes a_{1}\otimes\dots\otimes
		(-1)^{\tau} a_{i}a_{j}\otimes\dots\otimes a_{n-1})
		\\
		&=e_{\alpha}\otimes a_{1}\otimes\dots\otimes (-1)^{\tau}
		a_{i}a_{j}\otimes\dots\otimes 1\otimes\dots\otimes a_{n-1}
		\end{align*}
		where $1$ is in the position $j$. Then, it is enough to notice that
		\begin{equation*}
		e_{\alpha}\otimes a_{1}\otimes\dots\otimes (-1)^{\tau} a_{i}a_{j}\otimes\dots\otimes 1\otimes\dots\otimes a_{n-1}
		\end{equation*}
		is in the same equivalence class of
		$e_{\alpha}\otimes a_{1}\otimes\dots\otimes a_{n}$, since for
		$a$,$b\in A$,
		\begin{equation*}
		a\otimes b=(a\otimes 1)(1\otimes b)\sim(1\otimes a)(1\otimes b)=(1\otimes ab)
		\end{equation*}
		and
		\begin{equation*}
		(a\otimes 1)(1\otimes b)\sim(a\otimes 1)(b\otimes 1)=(ab\otimes 1).
		\end{equation*}
		
	\end{remark}
	
	\begin{remark}
		[\cite{B-S}] The differential
		$\partial:\mathcal{C}_{BS}^{k}\rightarrow\mathcal{C}_{BS}^{k+1}$
		induced by $\Phi$ on
		\begin{equation*}C_{BS}^{\ast}(\Gamma)=\bigoplus_{S\subseteq
			E(\Gamma)}e_{S}\otimes A^{l(S)}\end{equation*} is
		\begin{align*}
		\partial(e_{S}\otimes a_{1}\otimes\dots\otimes a_{l(S)})
		=\sum_{\alpha\in E(\Gamma),
			l(S\cup\alpha)=l(S)}e_{\alpha}e_{S}\otimes
		a_{1}\otimes\dots\otimes a_{l(S)}\\ +\sum_{\alpha\in
			E(\Gamma),
			l(S\cup\alpha)=l(S)-1}(-1)^{\tau}e_{\alpha}e_{S}\otimes
		a_{1}\otimes\dots\otimes a_{i}\cdot a_{j}\otimes\dots
		\otimes a_{l(S)}.
		\end{align*}

		The first term of the sum represents the case where the edge
		$\alpha$ connects two vertices of the edges in $S$ that are
		in the same component. Therefore, the number of components of
		$[\Gamma:S]$ and $[\Gamma:S\cup \alpha]$ are the same, so
		$l(S)=l(S\cup\alpha)$. The second term of the sum refers to
		the case where the edge $\alpha$ connects two different
		components, so $l(S\cup\alpha)=l(S)-1$. Suppose that
		$\alpha:i\rightarrow j$, and that $a_{h}$ is the term
		corresponding to the component containing $i$, and $a_k$ to the component containing $j$. Then $\tau$ is
		the Kozul sign given by the permutation in the tensor
		product that moves $a_{k}$ to the position immediate to the
		right of $a_{h}$.
		
	\end{remark}
	\begin{example}
		Let $\Gamma$ be $K_{3}$, the complete graph with $3$
		vertices. The order of the vertices induces an order on the
		edges given by the lexicographic order
		$E(K_3)=\{e_{0,1},e_{0,2}, e_{1,2}\}$ and an orientation on
		the edges.
		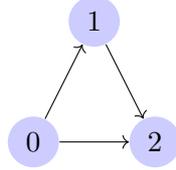
\begin{figure}[h!]
			\begin{center}
				\begin{tikzpicture}
				[scale=.8,auto=left,every node/.style={circle,fill=blue!20}]
				
				\node (n5) at (12,3) {1};
				\node (n1) at (13,1) {2};
				\node (n2) at (11,1)  {0};
				\foreach \from/\to in {n2/n1, n2/n5, n5/n1}
				\draw [->](\from) -- (\to);
				
				\end{tikzpicture}
				\caption{The graph $K_3$.}
			\end{center}
			
		\end{figure}
		
		The chain $\mathcal{C}_{BS}(\Gamma)$ in this case is the
		following
		\begin{equation*}
		A^{\otimes 3}\rightarrow A^{\otimes 2}\oplus A^{\otimes 2}\oplus A^{\otimes 2}\rightarrow A\oplus A\oplus A\rightarrow A
		\end{equation*}
		The chain groups are given by
		\begin{equation*}
		\mathcal{C}_{BS}(\Gamma)^{n}=\bigoplus_{|S|=n}A^{\otimes l(S)}
		\end{equation*}
		where $S$ is a subset of $E(\Gamma)$, and $l(S)$ is the number
		of components of $[\Gamma:S]$. The picture shows the
		components of $[\Gamma:S]$ with increasing cardinality of
		$S$ and the differential that adds every time an edge in
		$[\Gamma:S]$, connecting its components.
		Let $a_0\otimes a_1\otimes a_2\in \mathcal{C}_{BS}^0$, then
		\begin{align*}
		\MoveEqLeft \partial(a_0\otimes a_1\otimes a_2)
		\\
		&=e_{0,1}\otimes a_0a_1\otimes a_2+e_{1,2}\otimes
		a_0\otimes a_1a_2+(-1)^{|a_1||a_2|}e_{0,2}\otimes
		a_0a_2\otimes a_1
		\end{align*}
		Notice that the fact that $\partial^2=0$ is provided by the
		sign coming from the graded commutativity of $A$. For
		example,
		\begin{align*}
		\MoveEqLeft \partial_{e_{0,1}}((-1)^{|a_1||a_2|}e_{0,2}\otimes
		a_0a_2\otimes a_1)
		\\
		&=(-1)^{|a_1||a_2|}e_{0,1}e_{0,2}\otimes a_0a_2
		a_1=e_{0,1}e_{0,2}\otimes a_0a_1 a_2
		\end{align*}
		and
		\begin{equation*}
		\partial_{e_{0,2}}(e_{0,1}\otimes a_0a_1\otimes a_2)=e_{0,2}e_{0,1}\otimes a_0a_1a_2=-e_{0,1}e_{0,2}\otimes a_0a_1a_2.
		\end{equation*}
		All the terms in $\partial^2$ given by adding an edge $e_i$
		and then $e_j$ cancel with the terms given by adding the
		edges in opposite order, so $\partial^2=0$.
		
		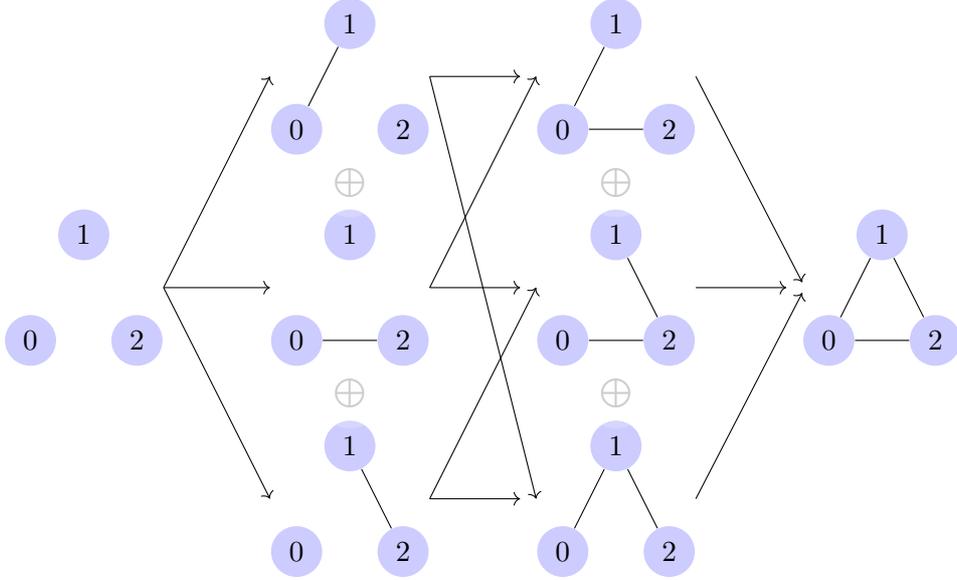
\begin{figure}[h!]
			\begin{tikzpicture}	
			[scale=.7,auto=left,every node/.style={circle,fill=blue!20}]
			
			\node (n1) at (1,2) {1};
			\node (n2) at (2,0) {2};
			\node (n0) at (0,0)  {0};

			\node (n1) at (6,6) {1};
			\node (n2) at (7,4) {2};
			\node (n0) at (5,4)  {0};
			\foreach \from/\to in {n0/n1}
			\draw (\from) -- (\to);
			
			\node (n1) at (6,2) {1};
			\node (n2) at (7,0) {2};
			\node (n0) at (5,0)  {0};
			\foreach \from/\to in {n0/n2}
			\draw (\from) -- (\to);
			
			\node (n1) at (6,-2) {1};
			\node (n2) at (7,-4) {2};
			\node (n0) at (5,-4)  {0};
			\foreach \from/\to in {n1/n2}
			\draw (\from) -- (\to);
			
			\node (n1) at (11,6) {1};
			\node (n2) at (12,4) {2};
			\node (n0) at (10,4)  {0};
			\foreach \from/\to in {n0/n1, n0/n2}
			\draw (\from) -- (\to);
			
			\node (n1) at (11,2) {1};
			\node (n2) at (12,0) {2};
			\node (n0) at (10,0)  {0};
			\foreach \from/\to in {n1/n2, n0/n2}
			\draw (\from) -- (\to);
			
			\node (n1) at (11,-2) {1};
			\node (n2) at (12,-4) {2};
			\node (n0) at (10,-4)  {0};
			\foreach \from/\to in {n0/n1, n1/n2}
			\draw (\from) -- (\to);

			\node (n1) at (16,2) {1};
			\node (n2) at (17,0) {2};
			\node (n0) at (15,0)  {0};
			\foreach \from/\to in {n0/n1, n0/n2, n1/n2}
			\draw (\from) -- (\to);
			
			\draw [->](2.5,1)->(4.5,5);
			\draw [->](2.5,1)->(4.5,1);
			\draw [->](2.5,1)->(4.5,-3);
			
			\draw [->](7.5,1)->(9.5,5);
			\draw [->](7.5,1)->(9.2,1);
			
			\draw [->](7.5,5)->(9.2,5);
			\draw [->](7.5,5)->(9.5,-3);
			
			\draw [->](7.5,-3)->(9.2,-3);
			\draw [->](7.5,-3)->(9.5,1);
			
			\draw [->](12.5,5)->(14.5,1.1);
			\draw [->](12.5,1)->(14.2,1);
			\draw [->](12.5,-3)->(14.5,0.9);

			\node [fill=white, circle,opacity=.2,text opacity=1]at (6,3){$\bigoplus$};
			\node [fill=white,opacity=.2,text opacity=1]at (6,-1){$\bigoplus$};
			\node [fill=white,opacity=.2,text opacity=1]at (11,3){$\bigoplus$};
			\node [fill=white,opacity=.2,text opacity=1]at (11,-1){$\bigoplus$};
			\end{tikzpicture}	
			\caption{An example of $\mathcal{C}_{BS}(\Gamma)$ when
				$\Gamma=K_3$, the complete graph with three vertices.}
		\end{figure}
		
	\end{example}

	As anticipated, the complex $\mathcal{C}_{BS}$ is related to
	the homology of configuration spaces depending on a graph, as
	defined by \emph{Eastwood} and
	\emph{Hugget} \cite{E-H}.
	
	Let $M$ be simplicial complex and $\Gamma$ a graph as defined
	in the first section. Let $\alpha:i\rightarrow j$ be an edge
	in $E(\Gamma)$, $Z_{\alpha}$ be the diagonal of the
	Cartesian product $M^{n}$ corresponding to the edge $\alpha$,
	\begin{equation*}
	Z_{\alpha}=\{(m_{1},\dots,m_{n})\in M^{n}; m_{i}=m_{j}\}
	\end{equation*}
	and
	\begin{equation*}
	Z_{\Gamma}=\bigcup_{\alpha\in E(\Gamma)}Z_{\alpha}.
	\end{equation*}
	We define the graph
	configuration space of $M$ dependent on $\Gamma$ to be
	\begin{equation*}
	\Conf(M,\Gamma)=M^{n}\backslash Z_{\Gamma}.
	\end{equation*}
	If $M$ is a manifold the definition corresponds to the
	generalized configuration space depending on a graph studied
	by \emph{Eastwood} and \emph{Hugget} in \cite{E-H}.
	
	\emph{Baranovsky} and \emph{Sazdanovi\'{c}} in \cite{B-S}
	prove that $\mathcal{C}_{BS}$ is the $E_{1}$ page of a
	spectral sequence converging to the cohomology of such
	configuration space. This confirms a conjecture by
	\emph{Khovanov} that there is a spectral sequence between the
	graph homology defined by \emph{L.Helme-Guizon} and
	\emph{Y. Rong} and the work by \emph{Eastwood} and
	\emph{Huggett}.
	\begin{teorema}
		[\cite{B-S}] Assume that the cohomology algebra $A=H^{*}(M,R)$
		is a projective $R$-module and that $\Gamma$ has no loops or
		multiple edges. There exist a spectral sequence with $E_{1}$
		term isomorphic to $\mathcal{C}_{BS}^{*}$ which converges to
		the relative cohomology $H^{*}(M^{n},Z_{\Gamma};R)$.
	\end{teorema}
	\begin{remark}[\emph{Remark }$4$ \cite{B-S}]
		When $M$ is a compact $R$-oriented manifold of dimension
		$m$, the relative cohomology groups
		$H^{*}(M^{n},Z_{\Gamma};R)$ are isomorphic to the homology
		groups $H_{nm-\ast}(\Conf(M,\Gamma);R)$ by Lefschetz duality.
	\end{remark}
	Moreover in the case where $M$ is a K\"{a}hler manifold the
	following result holds.
	\begin{remark}
		[\cite{B-S}]\label{remark3} If $M$ is a compact K\"{a}hler
		manifold and the coefficient ring $R$ is the rationals
		$\mathbb{Q}$ the spectral sequence degenerates at page
		$E_{2}$.
	\end{remark}

	\section{The Kriz model}\label{Kriz}
	In this section we describe the rational model for the configuration space
	of points in a complex projective variety defined by
	\emph{Kriz} in \cite{Kriz}.
	
	Let $X$ be a smooth projective variety over $\mathbb{C}$ and $\Conf(X,n)$
	be the ordered configuration space of $n$ points in a space $X$,
	\begin{equation*}
	\Conf(X,n)=X^{n}\setminus\bigcup_{i\neq
		j}\Delta_{i,j}
	\end{equation*}
	$\Delta_{i,j}=\{(x_{1},\dots, x_{n})\in X^{n}; x_{i}=x_{j}\}$. For
	$a$, $b\in\{1,\dots,n\}$, $a\neq b$, let
	$p_{a}^{\ast}:H^{\ast}(X)\rightarrow H^{\ast}(X^{n})$ the pullback of
	the projection $p_{a}:X^{n}\rightarrow X$ to the $a$-th coordinate and
	let $p_{a,b}^{\ast}:H^{\ast}(X^{2})\rightarrow H^{\ast}(X^{n})$ the
	pullback of the projection $p_{a, b}:X^{n}\rightarrow X^{2}$. Let
	$\Delta\in H^{\ast}(X^{2})$ be the class of the diagonal.
	\begin{teorema}[\cite{Kriz}]\label{kriz}
		Let $X$ be a complex projective variety of complex dimension
		$m$. Then the space $\Conf(X,n)$ has a model $E(n)$ that is isomorphic
		to
		\begin{equation*}
		H^{\ast}(X^{n},\mathbb{Q})[G_{a,b}]
		\end{equation*}
		where $G_{a,b}$ are generators of degree $2m-1$, $a$,
		$b\in\{1,\dots,n\}$, $a\neq b$ modulo the relations
		\begin{itemize}
			\item $G_{a,b}=G_{b,a}$
			\item $p_{a}^{\ast}(x)G_{a,b}=p_{b}^{\ast}(x)G_{a,b}$,
			$x\in H^{\ast}(X)$
			\item $G_{a,b}G_{b,c}+G_{b,c}G_{c,a}+G_{c,a}G_{a,b}=0$
		\end{itemize}
		The differential is given by $d(G_{a,b})=p_{a,b}^{\ast}\Delta$.
	\end{teorema}
	
	\goodbreak
	
	\begin{remark}
		The third relation $G_{a,b}G_{b,c}+G_{b,c}G_{c,a}+G_{c,a}G_{a,b}=0$
		is known in the literature as \emph{Arnold relation}.
	\end{remark}
	
	The definition of this graded algebra presents some similarities with
	the graded complex defined in \emph{Section }\ref{BS}: the structure
	of the exterior algebra with generators $G_{a,b}$ and the first two
	relations. However, the differential in $\mathcal{C}_{BS}(\Gamma)$
	"adds edges" while the one in $E[n]$ "removes edges". Therefore, we
	would like to relate the dual of the graded complex
	$\mathcal{C}_{BS}(\Gamma)$ with the DGA $E(n)$.

	Moreover, the complex $\mathcal{C}_{BS}(\Gamma)$ makes perfect sense
	in positive characteristic, so that we will also consider the
	following situation. Let $\kk$ be a ground ring, which could typically
	be $\ZZ$, $\QQ$, or a prime field $\field_p$. Assume that
	$A=H^*(X,\kk)$ is an algebra over $\kk$ which is free as a
	$\kk$-module. We extend the definition given by Kriz to this case by
	defining $E[n]$ as $A[G_{a,b}]/\sim$, where the relations are given by
	exactly the same three formulas as in the theorem above. It will be
	convenient to extend the definition further to the case where $A$ is a
	Frobenius algebra. To do this, we have to give a definition of
	$\Delta$ in this case, we do that in the next section.
	\section{Structures of tensor powers of Frobenius algebras}\label{frob}
	We will consider a graded version of Frobenius algebras. To be precise
	about how we understand that term in this chapter:
	\begin{definizione}
		A graded commutative Frobenius algebra $A$ over a commutative ground
		ring $\kk$ is a graded commutative ring, free and finite over
		$\kk=A^0$ as a module, together with a perfect, graded symmetrical
		pairing
		\begin{equation*}
		\innerproduct -- : A\otimes A\to \kk,
		\end{equation*}
		such that
		\begin{equation*}
		\innerproduct{ab}c=\innerproduct{a}{bc}.
		\end{equation*}
		
	\end{definizione}
	
	\begin{remark}
		Main example: Let $X$ be a compact, connected $\kk$-orientated
		manifold such that each cohomology group $H^i(X;\kk)$ is a free
		$\kk$-module. The cohomology ring $H^*(X,\kk)$ is a graded
		commutative Frobenius algebra over $\kk$. In this case, the pairing
		has degree $-\dim(X)$.
	\end{remark}
	
	If $A$ is a graded Frobenius algebra, so is $A\otimes A$. The
	multiplication is given by the usual tensor product of DGAs, involving
	the Koszul sign
	\begin{equation*}
	(a\otimes b)(c\otimes d)=(-1)^{|b|\cdot |c|}
	ac\otimes bd,
	\end{equation*}
	and the pairing is given by
	\begin{equation*}
	\innerproduct {a\otimes b}{c\otimes
		d}_2=(-1)^{|b|\cdot |c|} \innerproduct
	ac\innerproduct bd
	\end{equation*}
	where $|-|$ stands for the degree of the an element in the graded
	algebra.
	
	We can construct the dual $A^{\ast}=\hom(A,\kk)$ and we have an
	isomorphism of vector spaces given by
	\begin{equation*}
	k:A\cong A^{\ast}
	\end{equation*}
	\begin{equation*}
	a\mapsto  k(a)(-)=\innerproduct -a
	\end{equation*}
	$A$ is equipped with a multiplication $m:A\otimes A\rightarrow A$ and
	$A^{\ast}$ a dual map given by
	$m^{\ast}:A^{\ast}\rightarrow (A\otimes A)^{\ast}\cong A^{\ast}\otimes
	A^{\ast}$. Therefore we have a map $\mu^*: A\rightarrow A\otimes A$
	defined by composing the map $k$ that gives the isomorphism with the
	dual:
	\begin{equation*}
	\mu^*:A\xrightarrow{k}A^{\ast}\xrightarrow{m^*} A^{\ast}\otimes A^{\ast}\xrightarrow{k^{-1}\otimes k^{-1}} A\otimes A.
	\end{equation*}
	
	Alternatively, $\mu^*$ is defined by that
	\begin{equation*}
	\innerproduct {x\otimes y}{\mu^*(a)}_2=\innerproduct
	{xy}{a}
	\end{equation*}
	
	We see from this definition that $\mu^*:A\to A\otimes A$ is an
	$A\otimes A^{op}$ module map, since
	\begin{align*}
	\MoveEqLeft \innerproduct {x\otimes y}{(a\otimes 1)\mu^*(b)(1\otimes
		c)}_2
	\\
	&=
	(-1)^{|a|\cdot |xy|+|b|\cdot |c|}\innerproduct {(a\otimes 1)(x\otimes y)(1\otimes c)}{\mu^*(b)}_2\\
	&=(-1)^{|a|\cdot |xy|+|b|\cdot |c|}\innerproduct {axyc}{b}\\
	&=\innerproduct {xy}{abc}.
	\end{align*}

	We define $\Delta\in A\otimes A$ by the property
	\begin{equation*}
	\innerproduct{a\otimes
		b}\Delta_2=\innerproduct{ab}1.
	\end{equation*}
	
	\begin{remark}
		In the case $A=H^*(M,\kk)$ as considered above,
		$A\otimes A\cong H^*(M\times M,\kk)$, and $\Delta$ corresponds under
		this isomorphism to the Poincar\'e dual of the homology class of the
		diagonal $M\subset M\times M$.
	\end{remark}

	\begin{lemma}
		\label{lemma.delta}
		The class $\Delta$ satisfies that $(a\otimes b)\Delta=\mu^*(ab)$. In
		particular, $\mu^*: A\rightarrow A\otimes A$ is given by
		$\mu^*(a)=(a\otimes 1)\Delta$.
	\end{lemma}
	\begin{proof}
		Because the paring $\innerproduct -- _2$ is perfect, it suffices to
		prove that for any $x,y\in A$, we have that
		$\innerproduct{x\otimes y}{\mu^*(ab)}_2=\innerproduct {x\otimes
			y}{(a\otimes b)\Delta}_2$. We do the computation
		\begin{align*}
		\innerproduct{x\otimes y}{(a\otimes b)\Delta}_2&=
		\innerproduct{(x\otimes y)(a\otimes b)}{\Delta}_2\\
		&=\innerproduct{xyab}1\\
		&=\innerproduct {xy}{ab}\\
		&=\innerproduct{x\otimes y}{\mu^*(ab)}_2.
		\end{align*}
	\end{proof}
	
	\begin{remark}
		$\Delta$ has the property that
		$(1\otimes a)\Delta= (a\otimes 1)\Delta$, $a\in A$.
	\end{remark}

	We introduce some notation.  Let $S$ be a subset of the set of edges
	$E(\Gamma)$. Each $S$ determines a partition of the set of vertices so
	we have a map
	\begin{equation*}
	\Phi:\edges \Gamma\rightarrow\partitions \Gamma
	\end{equation*}
	where $\partitions \Gamma$ is the set of all partitions of $V(\Gamma)$
	and $\edges{E}$ the set of subsets of $E(\Gamma)$. The sets
	$\edges \Gamma$ form a partially ordered set by reversed  inclusion, and
	$\partitions \Gamma$ form partially ordered sets by refinement. The
	map $\Phi$ is order preserving.
	
	
	Note that the number $l(S)$ introduced in \emph{Definition }\ref{BS} corresponds to the number of sets in the partition $P=\Phi(S)$,
	that is the cardinality of $\Phi(S)$. We denote also by $|P|$ the
	number $l(S)$.
	
	There is a contravariant functor $\Psi$ from $\partitions \Gamma$ to
	graded algebras given by
	\begin{equation*}
	\Psi(P)=A^{\otimes |P| }.
	\end{equation*}
	
	The dual of the canonical surjective map
	$\Psi(P\to V(\Gamma)):A^{\otimes n} \to A^{\otimes |P|}$ is a
	canonical injective map
	\begin{equation*}
	A^{\otimes |P|}\cong (A^{\otimes
		|P|})^*\hookrightarrow (A^{\otimes n})^*\cong
	A^{\otimes n}.
	\end{equation*}
	
	For any partition $P$ we consider the image
	$\Delta_P\in A^{\otimes n}$ of $1\in A^{\otimes |P|}$. Using
	lemma~\ref{lemma.delta} inductively, we see that multiplying with this
	element is dual to the multiplication map in the sense that the
	following diagram commutes:
	\begin{equation*}
	\begin{tikzcd}[column sep=huge]
	A^{\otimes |P|}\ar[d,"k^{\otimes |P|}"]\ar[r,"\Delta_P\cdot"]&  A^{\otimes n}\ar[d,"k^{\otimes n}"]\\
	(A^{\otimes |P|})^*\ar[r,"\Psi(P\to V(\Gamma))^*"]& (A^{\otimes
		n})^*
	\end{tikzcd}
	\end{equation*}
	
	This element is invariant under any permutation in $S_n$ preserving
	$P$.  If $Q$ is a refinement of $P$, there is similarly a relative
	element $\Delta_{Q,P}\in A^{\otimes |Q|}$ such that the following
	diagram commutes:
	\begin{equation*}
	\begin{tikzcd}
	A^{\otimes |Q|}\ar[dr,"\Delta_Q"']\ar[rr,"\Delta_{Q,P}"]&  &A^{\otimes |P|}\ar[dl,"\Delta_P"]\\
	&A^{\otimes n}&
	\end{tikzcd}
	\end{equation*}
	Each algebra $A^{\otimes |P|}$ is a module over $A^{\otimes n}$, and
	multiplication by
	$\Delta_{P,Q}$ is a map of $A^{\otimes n}$-modules.
	
	%
	\section{The dual graded complex}\label{dual}

	Using the notation of the previous section, we can re-write
	$\cbs \Gamma$ as the graded chain complex
	
	\begin{equation*}
	\cbs \Gamma=\bigoplus_{S\subseteq E(\Gamma) }e_S\otimes A^{\otimes l(S)}= \bigoplus_{P\in\mathcal{P}}(\bigoplus_{S, \phi(S)=P}e_{S}\otimes A^{\otimes |P|}).
	\end{equation*}
	
	The differential is given by
	\begin{align*}
	\partial&=\sum_{e\in E(\Gamma)}\partial_e\\
	\partial_e(e_S\otimes x)&=(-1)^\tau e_{S\cup \{e\}}\otimes \Psi(\Phi(
	S\cup\{e\})).
	\end{align*}
	The sign $(-1)^\tau$ is determined by the number $\tau$ of edges in
	$S$ that precede $e$ in the chosen ordering of the edges.
	
	We note that as a graded vector space
	\begin{equation*}
	\mathcal{C}_{BS}(\Gamma)=\Lambda(e_{\alpha})\otimes (A^{\otimes n})/{\sim}
	\end{equation*}
	where $\sim$ indicates the relation $e_{\alpha}\otimes (a[i]-a[j])$, $a\in A$, $\alpha:i\rightarrow j\in E(\Gamma)$ and $a[i]$
	denotes the element
	$1^{\otimes i-1}\otimes a\otimes 1^{\otimes n-i}\in A^{\otimes n}$,
	described in the first section in \emph{Definition } \ref{def Cbs}.
	This relation corresponds to the second relation of the definition of
	the DGA defined by \emph{Kriz}, since
	$p_{a}(x)=1\otimes\dots\otimes x\otimes\dots\otimes 1\in A^{\otimes
		n}$ where $x$ is the $a$-th component of the tensor product.
	
	We want to describe the dual graded chain complex
	\begin{equation*}
	\mathcal{C}^\ast_{BS}(\Gamma)=\Bigl(\bigoplus_{S\subseteq E(\Gamma) }e_{S}\otimes A^{\otimes l(S)}\Bigr)^{\ast}.
	\end{equation*}

	We will denote by $G_{\alpha}$ the basis for the dual exterior algebra over  $e_{\alpha}$,
	for the edge $\alpha: i\rightarrow j$.
	We can so write the dual graded chain complex
	\begin{salign*}
	\mathcal{C}^\ast_{BS}( \Gamma)&=\Bigl(\bigoplus_{S\subseteq
			E(\Gamma)}e_{S}\otimes A^{\otimes l(S)}\Bigr)^{\ast}
		\\
		&=\bigoplus_{S\subseteq E(\Gamma)}(e_{S})^{\ast}\otimes
		(A^{\ast})^{\otimes l(S)}=\bigoplus_{S\subseteq
			E(\Gamma)}G_{S}\otimes A^{\otimes l(S)},
	\end{salign*}
where $G_S$ denotes the product of all the $G_{ij}$
where $\alpha: i\rightarrow j$ is an edge in $S$.

	The dual of the differential $\partial$, that we denote by $\delta$,
	acts by removing edges in the graph and therefore increasing the
	number of components. Let $G_ {S}$ be the product of all the $G_{ij}$
	where $\alpha: i\rightarrow j$ is an edge in $S$,
	\begin{equation*}
	\delta(G_ {S})=\sum _{i<j}(-1)^{\nu}\delta_{i,j}(G_ {S})=\sum_{i<j}(-1)^{\nu} G_{S\smallsetminus \alpha_{i,j}}
	\end{equation*}
	where $\nu$ is the number corresponding to the position of the edge
	$\alpha_{i,j}$ in the ascending order. We have
	\begin{equation*}
	\delta(G_ {S}\otimes a_{1}\otimes\dots\otimes a_{l(S)})=\sum _{i<j}(-1)^{\nu}\delta_{i,j}(G_ {S}\otimes a_{1}\otimes\dots\otimes a_{l(S)})
	\end{equation*}
	and
	\begin{salign*}
		\MoveEqLeft \delta_{i,j}(G_ {S}\otimes a_{1}\otimes\dots\otimes
		a_{l(S)})
		\\
		&=(-1)^{\tau} G_ {S\smallsetminus
			\alpha}\otimes(\Delta_{S,S\smallsetminus\alpha} \cdot a_{1}\otimes
		\dots\otimes a_{l(S)})
	\end{salign*}
	in the case $\alpha: i\rightarrow j$ is an edge belonging to $S$ and
	$l(S\smallsetminus\alpha)=l(S)-1$, and $\tau$ is the Kozul sign given by
	moving the factor in $\mu(a)$ in the $j$-th position. While
	\begin{equation*}
	\delta_{i,j}(G_ {S}\otimes a_{1}\otimes\dots\otimes a_{l(S)})=G_ {S\smallsetminus \alpha}\otimes a_{1}\otimes\dots\otimes a_{l(S)}
	\end{equation*}
	in the case $\alpha: i\rightarrow j$ is an edge belonging to $S$ and
	$l(S\smallsetminus\alpha)=l(S)$. Finally,
	\begin{equation*}
	\delta_{ij}(G_ {S}\otimes a_{1}\otimes\dots\otimes a_{l(S)})=0
	\end{equation*}
	if $\alpha$ does not belong to $S$.
	\begin{remark}
		We discuss here the grading of the dual complex. Let
		$S\subseteq E(\Gamma)$. We assign to an element
		$G_{S}\otimes a_{1}\otimes\dots\otimes a_{l(S)}$ in
		$\cbs{\Gamma}^{\ast}$ the grading
		\begin{equation*}
		(m-1)r_{\ext}-r_{\mathrm{int}}+\sum_{i}|a_{i}|,
		\end{equation*}
		where $m$ is the dimension of the manifold. $r_{\ext}$ is the number
		of external edges, that are the edges that, if removed, disconnect
		components, $r_{\mathrm{int}}$ the number of internal edges, that are the
		edges that do not disconnect components if removed and $|a_{i}|$ is
		the degree of the element $a_{i}$ in $A$. The differential has
		degree $1$ since
		\begin{salign*}
			\MoveEqLeft[0] \delta(G_{S}\otimes a_{1}\otimes\dots\otimes
			a_{l(S)})
			\\
			&=
			\begin{cases}
				0 & \text{if}\ \alpha\notin S\\
				\sum_{\alpha\in E(\Gamma)}G_{S\smallsetminus\alpha}\otimes\Delta_{S,S\smallsetminus\alpha}\cdot a_{1}\otimes\dots\otimes a_{l(S)}, & \text{if}\ \alpha \text{ disconnects}\ S \\
				\sum_{\alpha\in E(\Gamma)}G_{S\smallsetminus\alpha}\otimes
				a_{1}\otimes\dots\otimes a_{l(S)}, & \text{if}\ \alpha \text{
					non disconnects}\ S
			\end{cases}
		\end{salign*}
		and $\Delta_{S,S\smallsetminus\alpha}$ has degree $m$.  If $S$ is a
		forest, the grading of $\cbs{\Gamma}^{\ast}$ and the DGA
		$R(\Gamma, A)$ that we define later, coincide.
	\end{remark}
	\section{The ring $R(K_{n})$}\label{ring}
	In this section we want to study the ring defined by the exterior
	algebra $\Lambda[G_{a,b}]$, where $G_{a,b}$ are edges in a complete
	graph with $n$ vertices $K_n$, quotient by the relations introduced by
	Kriz in \emph{Theorem }\ref{kriz}.
	
	Let $K_n$ be a complete graph with $n$ vertices and $\Lambda[G_{a,b}]$
	be the exterior algebra with generators $G_{a,b}$ given corresponding
	to the edges in $K_n$. We define
	\begin{equation*}
	R(K_n)=\Lambda[G_{a,b}]/{\sim}
	\end{equation*}
	where $\sim$ is the Arnold relation
	$G_{i,j}G_{j,k}+G_{j,k}G_{k,i}+G_{k,i}G_{i,j}$. We call $I(K_n)$ the
	ideal generated by this relation, in order to simplify the notation
	this will be denoted also by $I$.  
	
	The following lemmas characterize
	the ideal $I$. We denote by $G_{\Gamma}$ the product of the generators
	corresponding to edges in $\Gamma$.
	\begin{lemma}\label{cycle}
		Denote by $v=(v_{1},\dots,v_{k})$, $k\geqslant3$ a set of vertices in
		$K_n$ and denote by $s(v)$ the product
		$s(v)=G_{v_{1},v_{2}}\cdot\ldots\cdot G_{v_{i},v_{i+1}}\cdot
		G_{v_{k},v_{1}}$ where $G_{v_{i},v_{i+1}}$ is the generator in the
		exterior algebra corresponding to the edge
		$\alpha:v_{i}\rightarrow v_{i+1}$, so $s(v)$ is the product of the
		generators corresponding to edges of a cycle of length $k$. Let $J$
		be the ideal generated by the elements $s(v)$ with
		$k\geqslant3$. Let $I$ be the ideal generated by $\delta(s(v))$ for every $v$ with
		$k\geqslant3$. Then $J$ is contained in $I$ and $I$ is generated by
		$\delta(s(v_{1},v_{2},v_{3}))=G_{v_{1},v_{2}}G_{v_{2},v_{3}}+G_{v_{2},v_{3}}G_{v_{3},v_{1}}+G_{v_{3},v_{1}}G_{v_{1},v_{2}}$.
	\end{lemma}
	\begin{proof}
		We first show that $J$ is contained in $I$.
		\begin{equation}
		\begin{split}
		\MoveEqLeft[3] G_{v_{1},v_{2}}\cdot\delta(s(v))
		\\
		= {} &G_{v_{1},v_{2}}\cdot(G_{v_{2},v_{3}}\cdot\ldots\cdot G_{v_{i},v_{i+1}}\cdot G_{v_{i+1},v_{i+2}}\cdot\ldots\cdot G_{v_{k},v_{1}})\\
		& -G_{v_{1},v_{2}}\cdot(G_{v_{1},v_{2}}\cdot G_{v_{3},v_{4}}\cdot\ldots\cdot G_{v_{i},v_{i+1}}\cdot G_{v_{i+1},v_{i+2}}\cdot\ldots\cdot G_{v_{k},v_{1}})\\
		&+\cdots\\
		= {} &G_{v_{1},v_{2}}\cdot G_{v_{2},v_{3}}\cdot\ldots\cdot
		G_{v_{i},v_{i+1}}\cdot G_{v_{i+1},v_{i+2}}\cdot\ldots\cdot
		G_{v_{k},v_{1}}=s(v).
		\end{split}
		\end{equation}
		Now we want to show that $I$ is generated by
		$\delta(s(v_{1},v_{2},v_{3}))=G_{v_{1},v_{2}}G_{v_{2},v_{3}}+G_{v_{2},v_{3}}G_{v_{3},v_{1}}+G_{v_{3},v_{1}}G_{v_{1},v_{2}}$. Let
		$I_{k}$ the ideal generated by $\delta(s(v))$ where
		$v=(v_{1},\dots,v_{l})$, $l\leqslant k$.  $I=\cup I_{k}$, we want to
		show by induction that $I_{k}=I_{3}$, $k\geqslant 3$ where $I_{3}$
		is generated by $\delta(s(v_{1},v_{2},v_{3}))$. It is obviously true
		for $k=3$. Suppose it true for $k-1$, we want to show that for every
		$s(v_{1},\dots,v_{k})$, $\delta(s(v_{1},\dots,v_{k}))\in I_{k-1}$.
		Consider
		$X=\delta(s(v_{k},v_{1},v_{2}))\delta(G_{v_{2},v_{3}}\cdot\ldots\cdot
		G_{v_{k-1},v_{k}})\in I_{3}=I_{k-1}$, we can expand the
		expression
		\begin{AllowDisplayBreaks}
			\begin{align*}
			X ={}& (G_{v_{1},v_{2}}\cdot
			G_{v_{2},v_{k}}-G_{v_{k},v_{1}}\cdot
			G_{v_{2},v_{k}}+G_{v_{k},v_{1}}\cdot G_{v_{1},v_{2}})\cdot
			\\*
			&\cdot\Bigl(\,\sum_{2\leq j\leq
				k-1}(-1)^{j}G_{v_{2},v_{3}}\cdot\ldots\cdot
			\widehat{G}_{v_{j},v_{j+1}}\cdot\ldots\cdot
			G_{v_{k-1},v_{k}}\Bigr)
			\\
			= {} &(G_{v_{1},v_{2}})\cdot\Bigl(\,\sum_{2\leq j\leq
				k-1}(-1)^{j}G_{v_{k},v_{2}}\cdot
			G_{v_{2},v_{3}}\cdot\ldots\cdot
			\widehat{G}_{v_{j},v_{j+1}}\cdot\ldots\cdot
			G_{v_{k-1},v_{k}}\Bigr)
			\\*
			&-(G_{v_{k},v_{1}})\cdot\Bigl(\,\sum_{2\leq j\leq
				k-1}(-1)^{j}G_{v_{k},v_{2}}\cdot
			G_{v_{2},v_{3}}\cdot\ldots\cdot
			\widehat{G}_{v_{j},v_{j+1}}\cdot\ldots\cdot
			G_{v_{k-1},v_{k}}\Bigr)
			\\*
			& +\sum_{2\leq j\leq k-1}(-1)^{j}G_{v_{k},v_{1}}\cdot
			G_{v_{1},v_{2}}\cdot G_{v_{2},v_{3}}\cdot\ldots\cdot
			\widehat{G}_{v_{j},v_{j+1}}\cdot\ldots\cdot G_{v_{k-1},v_{k}}
			\\
			= {}
			&(G_{v_{1},v_{2}})\cdot
			\begin{aligned}[t]
			(&-\delta(s(v_{k},v_{2},\dots,v_{k-1}))
			\\
			&+G_{v_{2},v_{3}}\cdot\ldots\cdot
			G_{v_{j},v_{j+1}}\cdot\ldots\cdot G_{v_{k-1},v_{k}})
			\end{aligned}
			\\*
			&
			-(G_{v_{k},v_{1}})
			\begin{aligned}[t]
			\smash[b]{\bigl(}&-\delta(s(v_{k},v_{2},\dots,v_{k-1}))
			\\
			&+G_{v_{2},v_{3}}\cdot\ldots\cdot
			G_{v_{j},v_{j+1}}\cdot\ldots\cdot G_{v_{k-1},v_{k}}\smash[t]{\bigr)}
			\end{aligned}
			\\*
			&+\delta(s(v_{k},v_{1},\dots,v_{k-1}))-G_{v_{1},v_{2}}\cdot
			G_{v_{2},v_{3}}\cdot\ldots\cdot
			G_{v_{j},v_{j+1}}
			\\
			&+G_{v_{k},v_{1}}\cdot
			G_{v_{2},v_{3}}\cdot\ldots\cdot G_{v_{j},v_{j+1}}
			\\
			= {}
			& -G_{v_{1},v_{2}}\cdot\delta(s(v_{k},v_{2},\dots,v_{k-1}))
			\\*
			&+G_{v_{k},v_{1}}\delta(s(v_{k},v_{2},\dots,v_{k-1}))+\delta(s(v_{k},v_{1},\dots,v_{k-1})).
			\end{align*}
		\end{AllowDisplayBreaks}
		Note that the third equality comes from the fact that
		\begin{salign*}
			\MoveEqLeft \delta(s(v_{k},v_{2},\dots,v_{k-1}))
			\\
			&=\sum_{2\leq j\leq
				k}(-1)^{j-1}G_{v_{k},v_{2}}\cdot\ldots\cdot
			\widehat{G}_{v_{j},v_{j+1}}\cdot\ldots\cdot G_{v_{k-1},v_{k}}
		\end{salign*}
		that equals the first term in the sum in the expression apart from the missing
		term
		$G_{v_{2},v_{3}}\cdot\ldots\cdot G_{v_{j},v_{j+1}}\cdot\ldots\cdot
		G_{v_{k-1},v_{k}}$. 
		
		Now,
		$\delta(s(v_{k},v_{1},\dots,v_{k-1}))=(-1)^{k}\delta(s(v_{1},v_{2},\dots,v_{k}))$,
		so we can write
		\begin{salign*}
			\MoveEqLeft
			\delta(s(v_{1},v_{2},\dots,v_{k}))
			\\
			&=-G_{v_{1},v_{2}}\cdot\delta(s(v_{2},v_{3},\dots,v_{k}))+G_{v_{k},v_{1}}\delta(s(v_{2},v_{3},\dots,v_{k}))-X.
		\end{salign*}
		We can conclude that
		$\delta(s(v_{1},v_{2},\dots,v_{k})\in I_{k-1}=I_{3}$. This end the
		proof by induction, so $I_{k}=I_{3}$ for all $k$ and so $I=I_{3}$.
	\end{proof}
	\begin{corollario}
		If the graph $K_n$ contains a cycle, then $G_{\Gamma}\in I$.
	\end{corollario}
	
	We conclude that every element in $R_n(\Gamma)$ where $\Gamma= K_n$
	can be written as a linear combination of the classes $G_{\Gamma'}$
	where $\Gamma'$ are graphs which do not contain any cycles. Such a
	graph is a disjoint union of trees, that is, it is a forest. However,
	these classes are not linearly independent in $R_n(\Gamma)$. Let $F$ denote a forest in $\Gamma$. We can
	rewrite the complex $R_n(\Gamma)$ as
	
	\begin{equation*}
	R_{n}(\Gamma)=\Lambda[G_{a,b}]/\sim=\bigoplus_{F}\bigotimes_{T\subset F}\mathbb{Z}[T]/{\sim}
	\end{equation*}
	where $\mathbb{Z}[T]$ is the free group generated by the trees.
	We have from a result by Vassilev \cite{Vassiliev} that $\mathbb{Z}[T]=\mathbb{Z}^{(n-1)!}$.

	\section{The generalised DGA}
	
	We want to extend the definition of $E[n]$ to a graded algebra dependent by any graph $\Gamma$ and
	defined on a Frobenius algebra over any ring.  In order to do so, we need to
	modify the ideal $I(K_n)$ and introduce the following definition:
	\begin{definizione}\label{R(gamma)}
		Let $\Gamma$ be a graph, and $(v_{1},\dots,v_{k})$ $k\geqslant3$ a
		set of vertices in $\Gamma$. We call a cycle $w$ a subset of the set
		of edges of $\Gamma$ of the form
		$\{(v_{1},v_{2}),\dots,(v_{i},v_{i+1}),\dots, (v_{k},v_{1})\}$. Let
		$\Lambda[G_{a,b}]$ be the exterior algebra with generators $G_{a,b}$
		corresponding to the edges $(a,b)$ in $\Gamma$. We denote by $G_{w}$
		the product
		$G_{w}=G_{v_{1},v_{2}}\cdot\ldots\cdot
		G_{v_{i},v_{i+1}}\cdot\ldots\cdot G_{v_{k},v_{1}}$.  We define
		\begin{equation*}
		R(\Gamma)=\Lambda[G_{a,b}]/{\sim}
		\end{equation*}
		where $\sim$ are the relations
		\begin{itemize}
			\item $G_{a,b}=G_{b,a}$
			\item $	\delta(G_{w})=\sum_{i}(-1)^{i}G_{v_{1},v_{2}}\cdot\ldots\cdot\hat{G}_{v_{i},v_{i+1}}\cdot\ldots\cdot G_{v_{k},v_{1}}=0$
		\end{itemize}
		for all the cycles $w$ in $\Gamma$.  We call
		\emph{generalised Arnold relations} the second set of relations and
		$I(\Gamma)$ the ideal generated by them.
	\end{definizione}
	\begin{remark}
		Note that by the results in the previous section, if $\Gamma=K_n$
		then $I(\Gamma)=I(K_n)$.
	\end{remark}
	\begin{lemma}\label{cycle general garph}
		If $v$ is a cycle in $\Gamma$ then $G_{v}\in I(\Gamma)$.
	\end{lemma}
	\begin{proof}
		Let $v$ be the cycle with edges
		$\{(v_{1},v_{2}),\dots,(v_{i},v_{i+1}),\dots, (v_{k},v_{1})\}$
		\begin{equation*}
		\begin{split}
		\MoveEqLeft[3]  G_{v_{1},v_{2}}\cdot d(G_{v})
		\\
		= {} &G_{v_{1},v_{2}}\cdot(G_{v_{2},v_{3}}\cdot\ldots\cdot G_{v_{i},v_{i+1}}\cdot G_{v_{i+1},v_{i+2}}\cdot\ldots\cdot G_{v_{k},v_{1}})\\
		& -G_{v_{1},v_{2}}\cdot(G_{v_{1},v_{2}}\cdot G_{v_{3},v_{4}}\cdot\ldots\cdot G_{v_{i},v_{i+1}}\cdot G_{v_{i+1},v_{i+2}}\cdot\ldots\cdot G_{v_{k},v_{1}})\\
		&+\cdots\\
		= {} &G_{v_{1},v_{2}}\cdot G_{v_{2},v_{3}}\cdot\ldots\cdot
		G_{v_{i},v_{i+1}}\cdot G_{v_{i+1},v_{i+2}}\cdot\ldots\cdot
		G_{v_{k},v_{1}}=G_{v}.
		\end{split}
		\end{equation*}
	\end{proof}
	\begin{corollario}
		If $\Gamma$ contains a cycle then $G_{\Gamma}\in I(\Gamma)$.
	\end{corollario}
	We can conclude that the elements in $R_{n}(\Gamma)$ are linear
	combinations of forests. We now define the generalized complex.
	\begin{definizione}\label{complex R}
		Let $M$ a compact, connected $\kk$-orientated manifold of even
		dimension $m$, $A=H^{\ast}(M,\kk)$ be a Frobenius algebra, where
		$\kk$ is the ground ring. Let $\Gamma$ be a graph with $n$ edges and
		$k$ cycles $w_{j}$ $j=0,\dots,k$.  We define the differential graded algebra
		\begin{equation*}
		R(\Gamma, A)=\Lambda[G_{a,b}]\otimes A^{\otimes n}/\sim
		\end{equation*}
		where $G_{a,b}$ are generators of degree $m-1$,
		$(a,b)\in E(\Gamma)$, and $\sim$ are the relations
		\begin{itemize}
			\item $G_{a,b}=G_{b,a}$
			\item $p_{a}^{\ast}(x)G_{a,b}=p_{b}^{\ast}(x)G_{a,b}$,
			$x\in H^{\ast}(X)$
			\item
			$\delta(G_{w_{j}})=\sum_{i}(-1)^{i}G_{v_{1,j},v_{2,j}}\cdot\ldots\cdot\hat{G}_{v_{i,j},v_{i+1,j}}\cdot\ldots\cdot
			G_{v_{h,j},v_{1,j}}=0$, for all $j=0,\dots,k$
		\end{itemize}
		The differential is given by
		\begin{equation*}
		d(G_{a,b})=p_{a,b}^{\ast}\Delta
		\end{equation*}
		here $\Delta$ is the class of the diagonal as described in
		\emph{Section }\ref{frob} and $p_{a,b}^\ast$ the pull back of the projection defined in \textit{Section }\ref{Kriz}.
	\end{definizione}

	\section{A quasi equivalence}\label{section quasi eq}
	
	Let $M$ be an even dimensional, compact, connected $\kk$-orientated
	manifold of dimension $m$, $A=H^{\ast}(M,\kk)$ a graded commutative
	Frobenius algebra and $\Gamma$ be any graph. Let
	$(\mathcal{C}^\ast_{BS}( \Gamma),\delta) $ be the dual complex defined in \emph{Section
	}\ref{dual} as
	$\mathcal{C}^\ast_{BS}( \Gamma)=\bigoplus_{S\subseteq E(\Gamma)}G_{S}\otimes A^{\otimes
		l(S)}$. We consider the generalized complex given in
	\emph{Definition }\ref{complex R}
	\begin{equation*}
	R(\Gamma, A)=\Lambda[G_{a,b}]\otimes A^{\otimes n}/{\sim}
	\end{equation*}
	where $\sim$ are the relations introduced in \emph{Definition
	}\ref{complex R} and $\Lambda[G_{a,b}]$ is the exterior algebra with
	generators given by the edges in $\Gamma$. We want to show that there
	is a quasi equivalence between $(\mathcal{C}^\ast_{BS}( \Gamma),\delta)$ and
	$(R(A,\Gamma), d)$.
	\begin{remark}\label{differential}
		The differential in $\cbs \Gamma^*$ can be written as
		\begin{equation*}
		\delta=\delta_{\mathrm{int}}+\delta_{\mathrm{ext}}
		\end{equation*}
		where $\delta_{\mathrm{int}}$ is the differential that removes internal
		edges, meaning edges such that if removed they don't disconnect
		components, and $\delta_{\mathrm{ext}}$ is the differential that removes
		external edges, that are the edges that if removed they disconnect
		components. By \emph{Lemma }\ref{cycle general garph} we have that
		$R(\Gamma)$ is given by linear combination of forests and therefore
		$d=\delta_{\mathrm{ext}}$.
	\end{remark}
	\begin{definizione}
		Let $S\subseteq E(\Gamma)$. We define the following map of graded
		groups:
		\begin{align*}
		F: \mathcal{C}^\ast_{BS}( \Gamma) &\to R(\Gamma,A)\\
		F(G_{S} \otimes x) &= [G_{S}\otimes x]_{\sim}.
		\end{align*}
		In order to simplify the notation we will write $G_{S}\otimes x$
		instead of $[G_{S}\otimes x]_{\sim}$.
	\end{definizione}
	\begin{lemma}
		The map $F$ is compatible with the differential.
	\end{lemma}
	\begin{proof}
		Let $G_{\Gamma'}\otimes x\in\mathcal{C}^\ast_{BS}( \Gamma)$, where $\Gamma'$ is a
		subgraph of $\Gamma$ and we denote by $G_\Gamma '$ the product of the generators corresponding to the edges in a graph $\Gamma'$. We want to check the commutativity of the
		following diagram.
		\begin{equation*}
		\xymatrix{
		\mathcal{C}^\ast_{BS}( \Gamma)\ar[r]^{F} \ar[d]_{\delta} & R(A,\Gamma)\ar[d]^{d} \\
		\mathcal{C}^\ast_{BS}( \Gamma)\ar[r]^{F}& R(A,\Gamma) }
		\end{equation*}
		We consider first the case where $\Gamma'$ does not contain any
		cycle. If $\Gamma'$ does not contain any cycle
		\begin{equation*}
		d\circ F(G_{\Gamma'}\otimes x)=d(G_{\Gamma'}\otimes x).
		\end{equation*}
		On the other hand by \emph{Remark }\ref{differential}
		\begin{equation*}
		F\circ \delta(G_{\Gamma'}\otimes x)=F\circ \delta_{\mathrm{ext}}(G_{\Gamma'}\otimes x)=d(G_{\Gamma'}\otimes x).
		\end{equation*}
		Now, suppose that $\Gamma'$ contains a cycle, that we denote by $S$,
		then
		\begin{equation*}
		d\circ F(G_{\Gamma'}\otimes x)=d(0)=0,
		\end{equation*}
		by definition of $F$. To prove the commutativity of the diagram we
		want to show that $F\circ \delta(G_{\Gamma'}\otimes x)=0$. By the
		previous remark,
		\begin{equation*}
		\delta(G_{\Gamma'}\otimes x)= \delta_{\mathrm{int}}+\delta_{\mathrm{ext}}(G_{\Gamma'}\otimes x),
		\end{equation*}
		the first summand is given by
		$\delta_{\mathrm{int}}(G_{\Gamma'}\otimes x)=\delta(G_{S\cup
			S'})G_{\Gamma'/S\cup S'}\otimes x$, where $S'$ is the graph given
		by the internal edges in $\Gamma'$ that are not in $S$. The
		differential $\delta_{\mathrm{int}}$ doesn't change the number of components
		and so it doesn't act on $x\in A^{\otimes l(\Gamma)}$.
		Now,
		\begin{equation*}
		\delta(G_{S\cup S'})=\delta(G_{S})G_{S'}+G_{S}\delta(G_{S'})\in  I(\Gamma)
		\end{equation*}
		by \emph{Lemma }\ref{cycle general garph} because $G_{S}$ and
		$\delta(G_{S})$ belongs to $I(\Gamma)$, so
		$F\circ \delta_{\mathrm{int}}(G_{\Gamma'}\otimes x)=0$. The second summand is
		\begin{equation*}
		\delta_{\ext}(G_{\Gamma'}\otimes x)=\delta_{\ext}(G_{\Gamma'/S\cup S'})G_{S}G_{S'}\otimes x'
		\end{equation*}
		where $x'\in A^{\otimes l(\Gamma')-1}$. The term belongs to the ideal
		$I(\Gamma)$ since $G_{S}\in I(\Gamma)$ and so
		\begin{equation*}
		F\circ \delta_{\ext}(G_{\Gamma'}\otimes x)=0.
		\end{equation*}
	\end{proof}
	\begin{teorema}\label{quasi eq}
		The map $F$ is a quasi equivalence.
	\end{teorema}
	\begin{proof}
		We want to introduce two filtrations on $\mathcal{C}^\ast_{BS}( \Gamma)$ and on
		$R(A,\Gamma)$, and prove that $F$ is compatible with them and that
		it induces a quasi equivalence on the filtration quotients.
		
		Let $\Gamma$ be a graph with $n$ vertices, $S$ be a subset of the
		set of edges $E(\Gamma)$. $S$ determines a partition of the set of
		vertices, so we have a map $\Phi:\mathcal{E}\rightarrow\mathcal{P}$
		where $\mathcal{P}$ is the set of all partitions of $V(\Gamma)$ and
		$\mathcal{E}$ the set of subsets of $E(\Gamma)$. As noted in
		\emph{Section }\ref{dual} we can rewrite the complex
		$\mathcal{C}^\ast_{BS}( \Gamma)$ as
		\begin{equation*}
		\mathcal{C}^\ast_{BS}( \Gamma) =\bigoplus_{P\in\mathcal{P}}(\bigoplus_{S, \phi(S)=P}G_{S}\otimes A^{\otimes |P|})
		\end{equation*}
		$|P|$ is the number of classes in the partition $P$.\\
		There is a filtration of $\cbs \Gamma (A)^*$ given by
		\begin{equation*}
		\mathcal{F}_{k}=\bigoplus_{P\in\mathcal{P}, |P|\geq k}   \Bigl(\,\bigoplus_{S, \phi(S)=P}G_{S}\otimes A^{\otimes |P|}\Bigr).
		\end{equation*}
		$\mathcal{F}_{k}$ is a subcomplex of $\cbs \Gamma ^*$ since the
		differential $\delta$ acts by removing edges and so increasing the
		number of components. So
		\begin{equation*}
		\mathcal{F}_{n}\subseteq\dots\subseteq \mathcal{F}_{k}\subseteq \mathcal{F}_{k-1}\subseteq \dots\subseteq \cbs\Gamma^{\ast}
		\end{equation*}
		and $|V(\Gamma)|=n$.
		
		Similarly we have a filtration on $R(A,\Gamma)$ in terms of
		partitions. Since
		\begin{equation*} R(\Gamma)
		=\bigoplus_{P\in\mathcal{P}}\bigoplus_{S,
			\phi(S)=P}G_{S}\Bigm/\bigoplus_{S, \phi(S)=P}G_{S}\cap I(\Gamma)
		\end{equation*} we can define
		\begin{equation*}
		\mathcal{F'}_{k}=\bigoplus_{P\in\mathcal{P}, |P|\geq k}\Bigl(\bigoplus_{S, \phi(S)=P}G_{S}\Bigm/\bigoplus_{S, \phi(S)=P}G_{S}\cap  I(\Gamma) \Bigr)\otimes A^{\otimes |P|}
		\end{equation*}
		as before $\mathcal{F'}_{k}$ is a subcomplex of $R(A,\Gamma)$ since
		the differential $d$ acts by removing edges and so increasing the
		number of components. So
		\begin{equation*}
		\mathcal{F'}_{n}\subseteq\dots\subseteq \mathcal{F'}_{k}\subseteq \mathcal{F'}_{k-1}\subseteq \dots\subseteq R(A,\Gamma)
		\end{equation*}
		and $|V(\Gamma)|=n$.  We want to show that $F$ is compatible with
		the filtrations, that is
		$F(\mathcal{F}_{k})\subseteq\mathcal{F'}_{k}$. This is clearly true
		since $F(G_{\Gamma}\otimes x)=G_{\Gamma}\otimes x$ is $\Gamma$ if a
		forest and $0$ otherwise.
		
		There are two short exact sequences given by inclusion and quotient
		map
		\begin{equation*}
		\mathcal{F}_{k-1}\longrightarrow\mathcal{F}_{k}\longrightarrow\mathcal{F}_{k-1}/\mathcal{F}_{k}
		\end{equation*}
		and
		\begin{equation*}
		\mathcal{F'}_{k-1}\longrightarrow\mathcal{F'}_{k}\longrightarrow\mathcal{F'}_{k-1}/\mathcal{F'}_{k}.
		\end{equation*}
		The last step of the proof consists in showing that for every $k$,
		\begin{equation*}
		F:\mathcal{F}_{k-1}/\mathcal{F}_{k}\longrightarrow \mathcal{F'}_{k-1}/\mathcal{F'}_{k}
		\end{equation*}
		is a quasi equivalence and then use the long exact sequences in
		homology induced from the short exact sequences to prove the result.
		Now,
		\begin{equation*}
		\mathcal{F}_{k-1}/\mathcal{F}_{k}=\bigoplus_{P\in\mathcal{P}, |P|= k-1}\Bigl(\bigoplus_{S, \phi(S)=P}G_{S}\otimes A^{\otimes |P|}\Bigr)
		\end{equation*}
		is determined by the partitions with exactly $k-1$ classes.
		
		Let $S$ be the maximal subset of $E(\Gamma)$ with respect to the inclusion that determines a partition of
		the set of vertices $P=\{P_{1},\dots,P_{l}\}$ and let
		$\Gamma_{i}^{S}$, $0\leq i\leq l$, be the connected subgraph of $S$
		corresponding to the element $P_{i}$ in the partition. By
		\emph{Lemma }\ref{remark2} we can rewrite
		$\mathcal{F}_{k-1}/\mathcal{F}_{k}$
		as
		\begin{equation*}
		\mathcal{F}_{k-1}/\mathcal{F}_{k}=\bigoplus_{P\in\mathcal{P}, |P|= k-1}\Bigl(\bigotimes_ {1\leq i\leq k-1}C_{\conn}(\Gamma_{i}^{S})\Bigr)\otimes A^{\otimes |P|}.
		\end{equation*}
		We define the chain complex $C_{\conn}(\Gamma)$ for connected graphs
		with $h$ vertices, and for every $0\leq i\leq h$,
		$C_{\conn}(\Gamma)^{i}$ is the free abelian group generated by all connected
		subgraphs of $\Gamma$ with $i$ edges. Let $S$ be a connected
		subgraph of $\Gamma$, the differential is given by
		\begin{equation*}
		d_{\conn}(S)=\sum_{e\in E(\Gamma)}(-1)^{\nu}(S\smallsetminus e)
		\end{equation*}
		where $\nu$ is the position of edge $e\in E(\Gamma)$ in ascending order. If $S\smallsetminus e$ is not connected $d_{\conn}(S)=0$.\\
		$\mathcal{F'}_{k-1}/\mathcal{F'}_{k}$ can be written as
		\begin{equation*}
		\mathcal{F'}_{k-1}/\mathcal{F'}_{k}=\bigoplus_{P\in\mathcal{P}, |P|= k-1} \Bigl(\bigotimes_ {1\leq i\leq k-1}\bigoplus_{j}C_{\conn}(\Gamma_{i,j})\Bigr)\otimes A^{\otimes |P|}
		\end{equation*}
		where now $\Gamma_{i,j}\subset \Gamma_i$ are spanning trees. In particular, we have that
		$C_{\conn}(\Gamma')= C_{\conn}(\Gamma)/I(\Gamma)$.  By the
		\emph{K\"{u}nneth formula}, the problem reduces to checking if
		\begin{equation*}
		q:C_{\conn}(\Gamma)\rightarrow C_{\conn}(\Gamma)/I(\Gamma)
		\end{equation*}
		is a quasi equivalence. Here by $I(\Gamma)$ we mean the subgroup
		given by $\alpha(I(\Gamma))$ and $\alpha$ is the isomorphism defined
		in \emph{Lemma }\ref{remark2}.  By \emph{Lemma }\ref{del-contr} the
		homology of $C_{\conn}(\Gamma)$ is concentrated in degree $n-1$.
		Now, $C_{\conn}(\Gamma)/I(\Gamma)$ is a complex concentrated in
		dimension $n-1$ by \emph{Remark }\ref{C/I}, that is the chain group
		generated by the trees and
		\begin{equation*}
		(C_{\conn}(\Gamma)/I(\Gamma))^{n-1}=C_{\conn}(\Gamma)^{n-1}/d_{\conn}(C_{\conn}(\Gamma)^{n})
		\end{equation*}
		because by \emph{Lemma }\ref{ideal}
		$d_{\conn}(C_{\conn}(\Gamma)^{n})=I(\Gamma)$. Since
		$C_{\conn}^{n-2}=0$, we have
		\begin{equation*}
		H_{n-1}(C_{\conn}(\Gamma))=C_{\conn}(\Gamma)^{n-1}/I(\Gamma)=H_{n-1}(C_{\conn}(\Gamma)/I(\Gamma))
		\end{equation*}
		and
		\begin{equation*}
		H_{i}(C_{\conn}(\Gamma))=H_{i}(C_{\conn}(\Gamma)/I(\Gamma))=0
		\end{equation*}
		for $i\neq n-1$. Then
		\begin{equation*}
		H_{\ast}(C_{\conn}(\Gamma))=H_{\ast}(C_{\conn}(\Gamma)/I(\Gamma)),
		\end{equation*}
		so $q:C_{\conn}(\Gamma)\rightarrow C_{\conn} (\Gamma)/I(\Gamma)$ is a
		quasi equivalence.

		Finally we consider the long exact sequences in homology
		\begin{equation*}
		\xymatrix{
			H_{i+1}(\mathcal{F}_{k-1}/\mathcal{F}_{k})\ar[r]\ar[d]^{\simeq}\ar[r]&H_{i}(\mathcal{F}_{k})\ar[r]\ar[d]_{F_{\ast}} &H_{i}(\mathcal{F}_{k-1})\ar[r]\ar[d]^{F_{\ast}}& H_{i}( \mathcal{F}_{k-1}/\mathcal{F}_{k}\ar[d]^{\simeq})\\
			H_{i+1}(\mathcal{F'}_{k-1}/\mathcal{F'}_{k})\ar[r]&H_{i}(\mathcal{F'}_{k})\ar[r]
			&H_{i}(\mathcal{F'}_{k-1})\ar[r]&
			H_{i}(\mathcal{F'}_{k-1}/\mathcal{F'}_{k}) }
		\end{equation*}
		
		Since $\mathcal{F}_{n}=\mathcal{F'}_{n}$, we have that
		$H_{i}(\mathcal{F}_{n})= H_{i}(\mathcal{F'}_{n})$ for every
		$i\geq 0$. We can then use the \emph{Five Lemma} and induction on
		$k$ with initial step given by
		$H_{i}(\mathcal{F}_{n})=H_{i}(\mathcal{F'}_{n})$.
		\begin{equation*}\resizebox{1\hsize}{!}{ \xymatrix{
				H_{i}(\mathcal{F}_{n-1}/\mathcal{F}_{n})\ar[r]\ar[d]^{\simeq}\ar[r]& H_{i-1}(\mathcal{F}_{n})\ar[r]\ar[d]_{=} & H_{i-1}(\mathcal{F}_{n-1})\ar[r]\ar[d]^{q}& H_{i-1}( \mathcal{F}_{n-1}/\mathcal{F}_{k}\ar[r]\ar[d]^{\simeq})& H_{i-2}(\mathcal{F}_{n})\ar[d]_{=} \\
				H_{i}(\mathcal{F'}_{n-1}/\mathcal{F'}_{n})\ar[r]&
				H_{i-1}(\mathcal{F'}_{n})\ar[r] &
				H_{i-1}(\mathcal{F'}_{n-1})\ar[r]&
				H_{i-1}(\mathcal{F'}_{n-1}/\mathcal{F'}_{n}) \ar[r] &
				H_{i-2}(\mathcal{F'}_{n}) }}
		\end{equation*}
		Therefore $H_{i}(\mathcal{F}_{k})\simeq H_{i}(\mathcal{F'}_{k})$ for
		every $k$ and $i$.  This concludes the proof that $F$ is a quasi
		equivalence.

	\end{proof}

	\begin{lemma}\label{remark2}
		Let $S$ be the maximal subset of $E(\Gamma)$ with respect to the inclusion that determines the partition of
		the set of vertices $P=\{P_{1},\dots,P_{l}\}$. Let
		$\Gamma_{i}^{S}$, $0\leq i\leq l$, be the connected subgraph of $S$
		corresponding to the element $P_{i}$ in the partition. The map
		\begin{salign*}
			\alpha:\qquad&
			\smashoperator[l]{\bigoplus_{P, |P|=
					k-1}}\bigoplus_{T, \phi(T)=P}G_{T}\otimes A^{\otimes
				|P|}
			\\
			&\longrightarrow\bigoplus_{P\in\mathcal{P}, |P|=
				k-1}\Bigl(\bigotimes_ {1\leq i\leq
				k-1}C_{\conn}(\Gamma^S_{i})\otimes A^{\otimes |P|}\Bigr)
		\end{salign*}
		defined for every $T\subset S$ by $$\alpha(G_T\otimes a_1\otimes\dots\otimes a_{l(T)})=\Gamma^T_{1}\otimes\dots\otimes\Gamma^T_{i}\otimes \dots\otimes \Gamma^T_{k-1}\otimes a_1\otimes\dots\otimes a_{l(T)}$$
		is a graded group isomorphism.
	\end{lemma}
	\begin{proof}
		Let $P$ be a partition with $k-1$ classes and
		consider the maximal subset $S\subseteq E(\Gamma)$, such that
		$\Phi(S)=P$. Consider a class $P_{i}$
		corresponding to a connected subgraph of $S$,
		$\Gamma_{i}$. Since the tensor product
		$A^{\otimes |P|}$ is not affected by the
		differential, we can reduce to building a map
		\begin{equation*}
		\alpha:\bigoplus_{T, \phi(T)=P}G_{T}\rightarrow\bigotimes_ {1\leq i\leq k-1}C_{\conn}(\Gamma_{i})
		\end{equation*}
		where $G_T$ is the element of $\Lambda[G_{S}]$ the exterior algebra
		with generators given by the edges in $S$. Note that if
		$T$ be a subgraph of $S$, it can be written
		as $T=\Gamma^T_{1}\cup\dots\cup \Gamma^T_{k-1}$, where
		$\Gamma^T_{i}$ are connected subgraphs of
		$\Gamma_{i}$. The map $\alpha$
		is a group isomorphism, since there is a
		bijection between elements of the base of
		$\bigoplus_{T, \phi(T)=P}G_{T}$ and
		elements
		$\Gamma^T_{1}\otimes\dots\otimes\Gamma^T_{i}\otimes \dots\otimes \Gamma^S_{k-1}$ of the base of
		$\bigotimes_ {1\leq i\leq
			k-1}C_{\conn}(\Gamma_{i})$. Moreover, the map preserve the grading since the degree in   $\bigoplus_{S, \phi(S)=P}G_{S}$ is given by the number of edges in $S$ as well as the degree in $\bigotimes_ {1\leq i\leq k-1}C_{\conn}(\Gamma_{i})$ by the definition of the complex $C_{\conn}(\Gamma)$.
	\end{proof}

	\begin{remark}\label{C/I}
		As a consequence of \emph{Lemma }\ref{cycle general
			garph} we have that
		$\alpha(I(\Gamma))\cap C_{\conn}$ contains the
		connected graphs with cycles, so the connected
		graphs that are not spanning trees. Therefore,
		$C_{\conn}^{i}/I(\Gamma)$ is trivial for all
		$i\neq n-1$.
	\end{remark}
	\begin{lemma}\label{ideal}
		Let $(C_{\conn}(\Gamma), d_{\conn})$ be the chain
		complex defined in the proof, $\Gamma$ a connected
		graph with $n$ vertices, then
		$d_{\conn}(C_{\conn}^{n})=I(\Gamma)\cap
		C_{\conn}^{n-1}$.
	\end{lemma}
	\begin{proof}
		$C_{\conn}^{n}$ is the free group generated by
		connected subgraph $S$ of $\Gamma$ with $n$
		edges. That is the image under the map $\alpha$ of
		the algebra
		\begin{equation*}
		\bigoplus_{S\subseteq E(\Gamma), |S|=n}G_{S}.
		\end{equation*}
		Since $\Gamma$ has $n$ vertices, $S$ must contain a
		cycle.  We call $C$ the cycle, and  the exterior
		algebra $G_{S}$ is given by the product of the edges
		in $C$ and the product of the rest of the edges, that we
		call $S'$. Now
		$d_{\conn}(S)=\sum_{e}d_{\conn,e}(C)S'$, since
		removing one edge in $S'$ will give a non connected
		graph. By \emph{Lemma }\ref{ideal} the ideal
		generated by $d_{\conn}(C^{n})$ is
		$\alpha(I(\Gamma))\cap C_{\conn}^{n-1}$.
	\end{proof}
	\begin{lemma}\label{del-contr}
		Let $\Gamma$ be a graph with $n$ vertices. The
		homology of the complex $C_{\conn}(\Gamma)$ is
		concentrated in degree $n-1$.
	\end{lemma}
	\begin{proof}
		Let $\Gamma$ be a connected graph with $n$ vertices
		and let $e$ be an edge in $E(\Gamma)$. We order the
		edges in $\Gamma$ so that $e$ is the last edge. We
		denote by $\Gamma\smallsetminus e$ the graph obtained
		from $\Gamma$ by deleting the edge $e$ and
		$\Gamma/e$ the graph obtained from $\Gamma$ by
		contracting the edge $e$.  We want to prove that we have a short exact
		sequence
		\begin{equation*}
		0\rightarrow C_{\conn}(\Gamma\smallsetminus e)\xrightarrow{\alpha} C_{\conn}(\Gamma)\xrightarrow{\beta} C_{\conn}(\Gamma/ e)\rightarrow0.
		\end{equation*}
		The map $\alpha$ is
		the inclusion of subgraphs and is injective since
		$\ker(\alpha)$ is given by graphs that are mapped to
		$0$, but these are the disconnected graph that are
		$0$ also in $C_{\conn}^{i}(\Gamma\smallsetminus e)$, so
		$\ker(\alpha)=0$. $\beta$ is the contraction of the
		edge $e$ and it is surjective since every element in
		$C_{\conn}^{i}(\Gamma/e)$ is the image of an element
		in $C_{\conn}^{i}(\Gamma)$, and if a graph is
		disconnected in $C_{\conn}^{i}(\Gamma/e)$ so it is in
		$C_{\conn}^{i}(\Gamma)$. Now we want to show that
		$\alpha$ and $\beta$ are chain maps, so that the
		squares in the following diagram commute.
		\begin{equation*}
		\xymatrix{
			0\ar[r]\ar[d]\ar[r]&C_{\conn}^{i+1}(\Gamma\smallsetminus e)\ar[r]^{\alpha}\ar[d]^{d_{\conn}} &C_{\conn}^{i+1}(\Gamma)\ar[r]^{\beta}\ar[d]^{d_{\conn}}& C_{\conn}^{i}(\Gamma/ e )\ar[r]\ar[d]^{d_{\conn}}& 0\ar[d]\\
			0\ar[r]&C_{\conn}^{i}(\Gamma\smallsetminus
			e)\ar[r]^{\alpha}
			&C_{\conn}^{i}(\Gamma)\ar[r]^{\beta}&
			C_{\conn}^{i-1}(\Gamma/e)\ar[r]& 0 }
		\end{equation*}
		The right and left square are clearly
		commutative. Consider the second square, let
		$S\in C_{\conn}^{i}(\Gamma\smallsetminus e)$ then
		\begin{equation*}
		\alpha(d_{\conn}(S))=\alpha\Bigl(\,\sum_{l\in E(S)}(-1)^{\nu_{l}}S\smallsetminus l\Bigr)=\sum_{l\in E(S)}(-1)^{\nu_{l}}S\smallsetminus l
		\end{equation*}
		and
		\begin{equation*}
		d_{\conn}(\alpha(S))=d_{\conn}(S)=\sum_{l\in E(S)}(-1)^{\nu_{l}}S\smallsetminus l.
		\end{equation*}
		Consider now the third square, let
		$S\in C_{\conn}^{i}(\Gamma)$,
		\begin{equation*}
		\beta(d_{\conn}(S))=\beta\Bigl(\,\sum_{l\in E(S)}(-1)^{\nu_{l}}S\smallsetminus l\Bigr)=\sum_{l\in E(S\smallsetminus e)}(-1)^{\nu_{l}}(S\smallsetminus l)/e
		\end{equation*}
		and
		\begin{equation*}
		d_{\conn}(\beta(S))=d_{\conn}(S/e)=\sum_{l\in E(S\smallsetminus e)}(-1)^{\nu_{l}}S/e\smallsetminus l,
		\end{equation*}
		since we ordered the edges in $\Gamma$ so that $e$
		is the last edge. Reordering the edges of $\Gamma$
		commutes with the differential so considering the
		chain where the edge $e$ is the last edge in
		$\Gamma$ doesn't affect the computation of the
		homology.  There are long exact sequences in
		homology and we proceed by induction.
		\begin{align*}		
		H_{i-1}(C_{\conn}(\Gamma/e))\longrightarrow {} &H_{i-1}(C_{\conn}(\Gamma\smallsetminus e))\longrightarrow H_{i-1}(C_{\conn}(\Gamma))\\
		\longrightarrow {} &H_{i-2}(
		C_{\conn}(\Gamma/e))\longrightarrow
		H_{i-2}(C_{\conn}(\Gamma\smallsetminus e))
		\end{align*}
		From the long exact sequence follows that if
		the homology of $C_{\conn}(\Gamma\smallsetminus e)$
		is concentrated in degree $n-1$ and the
		homology of $C_{\conn}(\Gamma/ e)$ is
		concentrated in degree $n-2$, that are the
		degrees represented by trees, then the
		homology of $C_{\conn}(\Gamma)$ is concentrated
		in degree $n-1$. We prove by induction on the
		number of edges in $\Gamma$ that the homology
		of $C_{\conn}(\Gamma)$ is concentrated in
		degree $n-1$, where $n$ is the number of
		vertices in $\Gamma$.  Let $\Gamma$ be a
		connected graph with one edge $e$ and two
		vertices, then $\Gamma\smallsetminus e$ is a
		disconnected graph and so
		$C_{\conn}(\Gamma\smallsetminus e)$ is
		trivial. $\Gamma/e$ is a graph with one vertex
		and no edges, the complex $C_{\conn}(\Gamma/e)$
		is concentrated in degree
		$0$. $H_{\ast}(C_{\conn}(\Gamma/e))$ is
		concentrated in degree $0$ that is $n-2$ and
		so $H_{\ast}(C_{\conn}(\Gamma))$ is
		concentrated in degree $1=n-1$. Now suppose by
		induction that the statement is true for any 
		graph $\Gamma$ with $|E(\Gamma)|< k$. We want
		to prove it for $|E(\Gamma)|=k$. Then
		$\Gamma\smallsetminus e$ is either disconnected or
		it is a connected graph with $k-1$ edges and
		$n$ vertices. In the first case the homology
		is trivial, in the second case the homology
		$H_{\ast}(C_{\conn}(\Gamma\smallsetminus e))$ is
		concentrated in degree $n-1$ by inductive
		hypothesis. $\Gamma/e$ is a graph with $k-1$
		edges and $n-1$ vertices, it can have a loop
		or a multiple edge, by \emph{Lemma
		}\ref{loops} the homology
		$H_{\ast}(C_{\conn}(\Gamma/ e))$ is either
		trivial or concentrated in degree $n-2$.
	\end{proof}
	\begin{remark}
		This lemma provides an alternative proof of
		the result given by \emph{Vassilev} in
		\cite{Vassiliev} regarding the homology of the
		complex of connected subgraphs of the complete
		graph. The proof could be already present in
		the literature but we are not aware of any
		references.
	\end{remark}
	\begin{lemma}\label{loops}
		Let $\Gamma$ be a graph. If $\Gamma$ contains
		a loop, then the homology groups
		$H_{i}(C_{\conn}(\Gamma))$ are trivial, and
		replacing a multiple edge by a single edge
		doesn't change the homology.
	\end{lemma}
	\begin{proof}
		Let $\Gamma$ be a graph with a loop $l$, then
		$\Gamma\smallsetminus l$ and $\Gamma/l$ are the
		same graph and so
		$H(C_{\conn}(\Gamma\smallsetminus
		l))=H(C_{\conn}(\Gamma/ l))$. By the long exact
		sequence we have that $H(C_{\conn}(\Gamma))=0$.
		Let now $\Gamma$ have a multiple edge
		$e$. Then $\Gamma\smallsetminus e$ is a graph with
		single edges and $\Gamma/l$ is a graph with a
		loop, and so $H(C_{\conn}(\Gamma/ e))=0$. By
		the long exact sequence we have that
		$H(C_{\conn}(\Gamma\smallsetminus
		e))=H(C_{\conn}(\Gamma))$.
	\end{proof}
	\begin{remark}
		The idea in the proof of \emph{Lemma
		}\ref{loops} is the same as the proof of the
		similar statement in \emph{Corollary }$3.2$ in
		\cite{H-G_arbitrary}.
		
	\end{remark}

	\section{The chain
		$\mathcal{C}^{\ast}_{BS}(\Gamma)/I(\Gamma)$}
	In this chapter we want to discuss the complex
	$\mathcal{C}_{BS}^{\ast}(\Gamma)/I(\Gamma)$ and
	show that it is isomorphic to $R(A,\Gamma)$.
	
	The DGA $R(A,\Gamma)$ is defined as
	$\Lambda[G_{a,b}]\otimes A^{\otimes n}$ quotient
	by the relations
	
	\begin{itemize}
		\item $G_{a,b}=G_{b,a}$
		\item
		$p_{a}^{\ast}(x)G_{a,b}=p_{b}^{\ast}(x)G_{a,b}$,
		$x\in H^{\ast}(X)$
		\item $\delta(w_{i})=0$
	\end{itemize}
	where $w_{i}$ is a cycle in $\Gamma$, and with
	differential
	\begin{equation*}
	d(G_{a,b})=p_{a,b}^{\ast}\Delta.
	\end{equation*}
	
	Let $\widetilde{R}(A,\Gamma)$ be
	$\Lambda[G_{a,b}]\otimes A^{\otimes n}$ quotient
	only by the first two relations, with the same
	differential $d$ and let
	$(\mathcal{C}_{BS}^{\ast}(\Gamma), \delta)$ be
	the complex defined in \emph{Section}
	\ref{dual}.
	\begin{remark} \label{difference}
		
		The complexes
		$(\mathcal{C}_{BS}^{\ast}(\Gamma),\delta)$ and
		$\widetilde{R}(A,\Gamma)$ are the same as
		vector spaces and they differ only for the
		differentials. The differential of the first
		complex is defined as
		\begin{align*}
		\MoveEqLeft \delta(G_{S}\otimes
		a_{1}\otimes\dots\otimes a_{l(S)})
		\\
		& =
		\begin{cases}
		0 & \text{if}\ \alpha\notin S\\
		\sum_{\alpha\in E(\Gamma)}G_{S\smallsetminus\alpha}p_{i,j}(\Delta)\otimes a_{1}\otimes\dots\otimes a_{n}, & \text{if}\ \alpha \text{ disconnects}\ S \\
		\sum_{\alpha\in
			E(\Gamma)}G_{S\smallsetminus\alpha}\otimes
		a_{1}\otimes\dots\otimes a_{n}, &
		\text{if}\ \alpha \text{ non disconnects}\
		S
		\end{cases}
		\end{align*}
		
		Here $\alpha:i\rightarrow j$ and
		$p_{i,j}(\Delta)$ is the pullback of the
		projection.  Note that
		$G_{S\smallsetminus\alpha}p_{i,j}(\Delta)\otimes
		a_{1}\otimes\dots\otimes
		a_{n}=G_{S\smallsetminus\alpha}\Delta_{S,S\smallsetminus\alpha}\otimes
		a_{1}\otimes\dots\otimes a_{l(S)}$, due to the
		relation in the definition of
		$\mathcal{C}_{BS}(\Gamma)$. The differential
		multiplies by the diagonal class only when
		removing the edge disconnects components. On
		the other hand, the differential in $\widetilde{R}(\Gamma, A)$,
		given by
		\begin{equation*}
		d(G_{a,b})=p_{a,b}^{\ast}\Delta,
		\end{equation*}
		always multiplies by the diagonal
		class. The differentials $\delta$ and $d$ can
		be therefore written as
		$\delta=\delta_{\mathrm{int}}+\delta_{\ext}$ and
		$d=d_{\mathrm{int}}+d_{\ext}$ (see \emph{Remark
		}\ref{differential}), and
		$d_{\ext}=\delta_{\ext}$. Removing an edge in a
		cycle leaves the number of components
		unchanged, while removing edges in a forest
		disconnects components, so from the point
		where $S$ is a spanning forest the two
		complexes are the same.
	\end{remark}

	We now prove that there is an isomorphism of chain
	complexes between
	$\mathcal{C}_{BS}^{\ast}(\Gamma)/I(\Gamma)$ and
	$R(A,\Gamma)$.
	\begin{teorema}
		There is an isomorphism of chain complexes,
		\begin{equation*}
		\id:(\mathcal{C}_{BS}^{\ast}(\Gamma),\delta)/I(\Gamma)\rightarrow (R(A,\Gamma), d)
		\end{equation*}
	\end{teorema}
	\begin{proof}
		Consider the identity map
		$\id:\mathcal{C}_{BS}^{\ast}(\Gamma)/I(\Gamma)\rightarrow
		R(A,\Gamma)$. By \emph{Remark}
		\ref{difference}, the complexes
		$\mathcal{C}_{BS}^{\ast}(\Gamma)$ and
		$\widetilde{R}_n(A,\Gamma)$ differ only from
		the internal differential since the two
		differentials agree in the case where the
		subgraph $S$ of the set of vertices
		$E(\Gamma)$ is a forest. By \emph{Lemma}
		\ref{cycle general garph} the third relation
		assures that elements in the complexes
		corresponding to graphs $S$ with a component
		containing cycles are zero. Therefore
		$\mathcal{C}_{BS}^{\ast}(\Gamma)/I(\Gamma)$
		and $R(A,\Gamma)$ agree as vector spaces, as the two differentials $\delta=d$,
		since $d_{\mathrm{int}}=\delta_{\mathrm{int}}=0$.
	\end{proof}

	\newpage
	\bibliographystyle{plain}
	{\bibliography{progress}}
\end{document}